\documentclass[11pt,a4paper]{article}

\usepackage{amsmath,amssymb,amsthm,mathtools}
\usepackage{geometry}
\usepackage[colorlinks=true,linkcolor=black,citecolor=black,urlcolor=black]{hyperref}
\usepackage{enumitem}
\usepackage[numbers,sort&compress]{natbib}
\newtheorem{theorem}{Theorem}[section]
\newtheorem{proposition}[theorem]{Proposition}
\newtheorem{lemma}[theorem]{Lemma}
\newtheorem{corollary}[theorem]{Corollary}
\theoremstyle{definition}
\newtheorem{definition}[theorem]{Definition}

\newtheorem{principle}[theorem]{Principle}
\theoremstyle{remark}
\newtheorem{remark}[theorem]{Remark}

\newcommand{\R}{\mathbb{R}}
\newcommand{\norm}[1]{\lVert #1 \rVert}
\newcommand{\diag}{\mathrm{diag}}
\newcommand{\Tr}{\mathrm{Tr}}
\renewcommand{\div}{\mathrm{div}}
\DeclareMathOperator{\sech}{sech}

\newenvironment{keywords}{\smallskip\noindent\textbf{Keywords.}\enspace}{\par}

\title{Activation Saturation and Floquet Spectrum Collapse\\
       in Neural ODEs}
\author{Nikolaos M.\ Matzakos\thanks{%
  School of Pedagogical \&\ Technological
  Education (ASPETE), 141\,21 Athens, Greece.
  Currently visiting the Department of Mathematics,
  Chair of Dynamics, Control, Machine Learning and Numerics,
  Friedrich-Alexander-Universit\"at Erlangen--N\"urnberg (FAU),
  91058 Erlangen, Germany.
  E-mail: \texttt{nikmatz@aspete.gr}, \texttt{nikolaos.matzakos@fau.de}.
  ORCID: \texttt{0000-0001-8647-6082}.}}
\date{}

\begin{document}
\maketitle

\begin{abstract}
We prove that activation saturation imposes a structural
dynamical limitation on autonomous Neural ODEs
$\dot{h}=f_\theta(h)$ with saturating activations ($\tanh$,
sigmoid, etc.):
if $q$ hidden layers of the MLP $f_\theta$ satisfy
$|\sigma'|\le\delta$ on a region~$U$,
the input Jacobian is attenuated as
$\norm{Df_\theta(x)}\le C(U)$
(for activations with $\sup_{x}|\sigma'(x)|\le 1$,
e.g.\ $\tanh$ and sigmoid, this reduces to $C_W\delta^q$),
forcing every Floquet (Lyapunov) exponent
along any $T$-periodic orbit $\gamma\subset U$
into the interval $[-C(U),\;C(U)]$.
This is a collapse of the Floquet spectrum:
as saturation deepens ($\delta\to 0$),
all exponents are driven to zero,
limiting both strong contraction and chaotic sensitivity.
The obstruction is structural --- it constrains
the learned vector field at inference time,
independent of training quality.
As a secondary contribution, for activations with $\sigma'>0$,
a saturation-weighted spectral factorisation yields
a refined bound $\widetilde{C}(U)\le C(U)$
whose improvement is amplified exponentially in~$T$
at the flow level.
All results are numerically illustrated on the Stuart--Landau oscillator;
the bounds provide a theoretical explanation for the
empirically observed failure of $\tanh$-NODEs
on the Morris--Lecar neuron model.
\end{abstract}



\section{Introduction}
\label{sec:intro}

\paragraph{Background.}
Neural ODEs~\cite{Chen2018NeuralODE} model continuous-time dynamics
via $\dot{h}=f_\theta(h)$, $h(t_0)=h_0\in\R^d$,
where $f_\theta$ is an MLP; extensions to
input-driven settings include Neural CDEs~\cite{Kidger2020NeuralCDE}.
A fundamental question is: \emph{what dynamical behaviours can
such a system reproduce?}

\paragraph{The problem.}
For saturating activations ($|\sigma'(x)|\to 0$ as $|x|\to\infty$),
the MLP Jacobian factors as
$Df_\theta(x) = W_L\,D_{L-1}(x)\,W_{L-1}\cdots D_1(x)\,W_1$
with $D_k = \diag(\sigma'(a_{k,1}),\ldots)$
(Proposition~\ref{prop:jacobian}),
so saturation in $q$ layers gives
$\norm{Df_\theta(x)}\le C_W\delta^q\to 0$.
Unlike the vanishing-gradient problem~\cite{Bengio1994,Glorot2010},
which affects $\nabla_\theta\mathcal{L}$ during training
and can be mitigated by batch normalisation~\cite{IoffeS2015},
the attenuation of $Df_\theta(x)$ is a property of the
\emph{learned vector field at inference time}.
This is a \emph{structural limitation on the learned dynamics},
governed by the saturation depth~$q$ and the threshold~$\delta$.

\paragraph{Central thesis.}
The main message of this paper is:

\begin{principle}[Floquet spectrum collapse, informal]
\label{prin:sap}
Activation saturation attenuates the Jacobian of the learned
vector field.
Via the Liouville--Abel--Jacobi identity, this forces all Floquet
exponents of any periodic orbit within the saturated region toward
zero, collapsing the dynamical spectrum.
The achievable contraction and expansion rates are bounded
by a quantity that vanishes as $\delta\to 0$.
This is made rigorous in
Theorems~\textup{\ref{thm:main}}, \textup{\ref{thm:liouville}},
and~\textup{\ref{thm:individual}}.
\end{principle}

\paragraph{Motivating observation.}
The starting point for the present work is a concrete empirical puzzle:
in companion experiments with the Morris--Lecar neuron
model~\cite{MatzakosSfyrakis2026}, a $\tanh$-NODE systematically
failed to reproduce limit-cycle oscillations across all three
bifurcation regimes (Hopf, saddle-node on limit cycle, and
homoclinic), while replacing $\tanh$ with SiLU resolved the
failure in every case.
This activation-dependent dichotomy suggested that saturation
imposes a structural limitation on the vector field's
Jacobian --- one that persists at inference time, independent
of training quality.
The theorems below make this hypothesis precise.

\paragraph{Main contributions.}
We present our results in a clear hierarchy:

\begin{enumerate}[label=\textup{(\arabic*)},leftmargin=*]
\item \textit{Mechanism: Jacobian attenuation}
(Theorem~\ref{thm:main}).
If $q$ hidden layers are $\delta$-saturated on a convex
region~$U$, then $\norm{Df_\theta(x)}\le C(U)=C_W\delta^q
\prod_{k\notin\mathcal{K}}M_k(U)$ for all $x\in U$.
This is the underlying mechanism that drives all subsequent
dynamical consequences.

\item \textit{Main result: Floquet spectrum collapse}
(Theorems~\ref{thm:liouville} and~\ref{thm:individual}).
The Liouville--Abel--Jacobi identity converts
Jacobian attenuation into a dynamical obstruction:
every Floquet multiplier of a $T$-periodic orbit
$\gamma\subset U$ satisfies
$e^{-C(U)T}\le|\mu_i|\le e^{C(U)T}$,
and every Floquet (Lyapunov) exponent satisfies
$|\lambda_i|\le C(U)$.
Since $C(U)\to 0$ as $\delta\to 0$,
the entire Floquet spectrum collapses to zero.
This is a quantitative expressivity limitation:
it does not exclude periodic orbits, but bounds
the achievable contraction and expansion rates.

\item \textit{Refinement: saturation-weighted bounds}
(Section~\ref{sec:refined}).
For activations with $\sigma'>0$, redistributing
$D_k^{1/2}$ to adjacent weight matrices yields a
refined bound $\widetilde{C}(U)\le C(U)$, sharp
(Proposition~\ref{prop:sharp}), with improvement
amplified exponentially in~$T$ at the flow level
(Corollary~\ref{cor:exp_amp}).
\end{enumerate}

\paragraph{Related work.}
{\tolerance=1000\emergencystretch=2em
Neural ODE expressivity:
\cite{Dupont2019AugmentedNODE,Massaroli2020NODE};
activation effects on training: \cite{GaoNeuralODE2025}
(complementary---we treat inference).
The trace $\Tr(Df_\theta)$ drives density estimation
in FFJORD~\cite{Grathwohl2019FFJORD};
our bounds show saturation forces this trace to zero.
Finlay et al.~\cite{Finlay2020HowToTrain}
regularise $\norm{Df_\theta}$ during training;
saturation imposes analogous attenuation structurally.
Li et al.~\cite{LiLiuLiveraniZuazua2026} prove universal
approximation for semi-autonomous Neural~ODEs at rate
$\mathcal{O}(1/\sqrt{P})$ with ReLU, confirming that our
obstruction is specific to the combination of autonomy
and saturation.\par}

\paragraph{Classical dynamical systems theory.}
The theoretical tools employed in this paper ---%
the Liouville--Abel--Jacobi identity, Floquet multipliers,
and Gr\"onwall estimates for the variational equation ---%
belong to the classical qualitative theory of ordinary
differential equations, originating with
Floquet~\cite{Floquet1883} and developed in
Coddington--Levinson~\cite{CoddingtonLevinson1955},
Hartman~\cite{Hartman1964}, Chicone~\cite{Chicone2006},
and Guckenheimer--Holmes~\cite{GuckenheimerHolmes1983}.
What distinguishes the present work from this literature
is not the analytical machinery, which is entirely
standard, but the application context: rather than
analysing a given ODE, we analyse the entire
\emph{parametrised class} of vector fields learnable by
a neural network under saturation, and establish that
a single architectural property --- activation saturation
--- imposes a uniform Floquet spectral constraint on all
orbits of all networks in this class.

\paragraph{Expressivity of neural networks.}
Universal approximation theorems~\cite{Cybenko1989,Hornik1990}
establish that sufficiently wide or deep networks can
approximate any continuous function to arbitrary accuracy
on a compact set.
Capacity bounds~\cite{Barron1993} quantify
the approximation error as a function of network width,
while depth separation results~\cite{Montufar2014}
identify function classes for which depth confers an
exponential representational advantage.
However, all of these results concern
\emph{static function approximation}; they make no
claim about the dynamical properties of the flow
generated by an autonomous ODE with a learned
right-hand side.
The present paper introduces a complementary notion,
which we term \emph{dynamical expressivity}: whether
the flow of $\dot{h}=f_\theta(h)$ can reproduce
specified stability properties of periodic orbits.
This is a fundamentally different question, and the
obstruction we identify has no analogue in classical
approximation theory --- a network that universally
approximates the vector field pointwise may still fail
to reproduce the Floquet structure of its orbits when
activation saturation is present.

\paragraph{Physics-informed and structure-preserving learning.}
A complementary line of research embeds structural
constraints directly into learned models.
Physics-informed neural networks
(PINNs)~\cite{Raissi2019} penalise residuals of
governing equations during training;
Matzakos and Sfyrakis~\cite{MatzakosSfyrakis2026}
demonstrate that such structural bias substantially
improves accuracy on the Morris--Lecar model across
bifurcation regimes.
The approach of the present paper is distinct: rather
than incorporating known equations into the loss, we
characterise the structural limitations that activation
saturation imposes on the learned model class,
independent of any specific training objective.
This characterisation naturally suggests
embedding dynamical structure as an inductive bias
into the training procedure, a direction we identify
as an open problem for future work.

\paragraph{Proof strategy.}
Theorem~\ref{thm:main}: submultiplicativity applied to
$Df_\theta=W_L D_{L-1}\cdots D_1 W_1$.
Theorems~\ref{thm:liouville}--\ref{thm:individual}:
the Liouville--Abel--Jacobi identity
$\ln\det M_\gamma=\int_0^T\Tr(Df_\theta(\gamma(t)))\,dt$
combined with $|\Tr(A)|\le d\norm{A}$, plus
a Gr\"onwall argument for individual multipliers.
Section~\ref{sec:refined}: replace $D_k$ by $D_k^{1/2}D_k^{1/2}$
and redistribute to adjacent weight matrices before applying
submultiplicativity.

\paragraph{Organisation.}
\S\ref{sec:setting}: notation.
\S\ref{sec:attenuation}: Jacobian bound (mechanism).
\S\ref{sec:floquet}: Floquet--Liouville obstruction (main result).
\S\ref{sec:numerics}: numerical verification.
\S\ref{sec:refined}: refined bounds (secondary contribution).
\S\ref{sec:discussion}: discussion.
Appendix~\ref{app:auxiliary}: auxiliary results.
Appendix~\ref{app:numerics}: additional numerical experiments.

\section{Setting and Notation}
\label{sec:setting}

\subsection*{Notation.}
We write $\norm{A}$ for the \emph{operator} (spectral) norm of a matrix
$A\in\R^{m\times n}$, i.e.\ $\norm{A}:=\sup\{\norm{Av}_2:\norm{v}_2=1\}$.
For a diagonal matrix $D=\diag(d_1,\ldots,d_n)$,
$\norm{D}=\max_i|d_i|$.
All norms on $\R^d$ are the Euclidean norm.

\subsection*{Key distinction.}
Throughout, $\theta\in\R^p$ denotes the fixed weight vector after training.
We study the \emph{input Jacobian}
\[
  Df_\theta(x) := \frac{\partial f_\theta}{\partial x}\in\R^{d\times d},
  \qquad x\in\R^d,\;\theta\text{ fixed},
\]
which measures how the vector field responds to changes in the \emph{state}.
This is distinct from the \emph{parameter gradient}
$\nabla_\theta\mathcal{L}=\partial\mathcal{L}/\partial\theta$,
which measures how the loss responds to changes in the \emph{weights}
(the training-time vanishing gradient phenomenon
studied by~\cite{Bengio1994} concerns the latter, not the former).
For a vector field $f_\theta:\R^d\to\R^d$, the divergence satisfies
\begin{equation}
  \div(f_\theta)(x) = \sum_{i=1}^d\frac{\partial(f_\theta)_i}{\partial x_i}(x)
  = \Tr\!\bigl(Df_\theta(x)\bigr).
  \label{eq:div_trace}
\end{equation}

\subsection*{Neural ODE.}
\begin{equation}
  \dot{h}(t) = f_\theta(h(t)),\qquad h(t_0)=h_0\in\R^d,
  \label{eq:node}
\end{equation}
where $f_\theta:\R^d\to\R^d$ is the $L$-layer fully connected network
\begin{equation}
  f_\theta(x)
  = W_L\,\sigma\!\bigl(W_{L-1}\,\sigma(\cdots\sigma(W_1 x+b_1)\cdots)
    +b_{L-1}\bigr)+b_L,
  \label{eq:mlp}
\end{equation}
with $W_\ell\in\R^{d_\ell\times d_{\ell-1}}$, $b_\ell\in\R^{d_\ell}$,
$d_0=d_L=d$, and $\sigma\in C^1(\R)$ applied componentwise.
We write
\[
  z_0(x):=x,\qquad
  a_k(x):=W_k z_{k-1}(x)+b_k,\qquad
  z_k(x):=\sigma(a_k(x)),
  \quad k=1,\ldots,L-1,
\]
for the pre-activation and post-activation vectors at layer $k$,
and $f_\theta(x)=W_L z_{L-1}(x)+b_L$.
We denote the flow $\Phi^\theta_t$, so $h(t)=\Phi^\theta_t(h_0)$.

\paragraph{Notational convention.}
We use $x$ as a generic argument of the MLP $f_\theta$ when stating
static bounds (Jacobian factorisation, operator-norm estimates), and
$h(t)$ or $\gamma(t)$ when the argument is a point on a trajectory
or periodic orbit.
Since the ODE $\dot{h}=f_\theta(h)$ evaluates the same map $f_\theta$,
any bound proved ``for all $x\in U$'' applies in particular to
$h(t)\in U$ along the flow.

\subsection*{MLP Jacobian.}

\begin{proposition}[Jacobian factorization; standard, see e.g.\ {\cite{Goodfellow2016}}]
\label{prop:jacobian}
For every $x\in\R^d$,
\begin{equation}
  Df_\theta(x)
  = W_L\,D_{L-1}(x)\,W_{L-1}\cdots D_1(x)\,W_1,
  \label{eq:jacobian}
\end{equation}
where
$D_k(x):=\diag\!\bigl(\sigma'(a_{k,1}(x)),\ldots,\sigma'(a_{k,d_k}(x))\bigr)
\in\R^{d_k\times d_k}$.
\end{proposition}

The proof is a direct induction on the chain rule;
see e.g.\ \cite{Goodfellow2016}.

Since $D_k(x)$ is diagonal, its operator norm satisfies
\begin{equation}
  \norm{D_k(x)} = \max_{1\le i\le d_k}\lvert\sigma'(a_{k,i}(x))\rvert.
  \label{eq:Dk_norm}
\end{equation}
Applying submultiplicativity of the operator norm to~\eqref{eq:jacobian}:
\begin{equation}
  \norm{Df_\theta(x)}
  \le \underbrace{\prod_{\ell=1}^L\norm{W_\ell}}_{=:\,C_W}
  \cdot\prod_{k=1}^{L-1}\max_{i}\lvert\sigma'(a_{k,i}(x))\rvert.
  \label{eq:jac_bound}
\end{equation}

\subsection*{Saturation.}

\begin{definition}
\label{def:saturation}
Let $U\subset\R^d$ be convex and $\mathcal{K}\subseteq\{1,\ldots,L-1\}$
with $|\mathcal{K}|=q$.
We call $q$ the \emph{saturation depth}: the number of hidden layers
whose activation derivatives are uniformly small on $U$.
\emph{$q$ layers are $\delta$-saturated on $U$} if
\begin{equation}
  \max_{1\le i\le d_k}\lvert\sigma'(a_{k,i}(x))\rvert\le\delta,
  \qquad\forall\,x\in U,\;\forall\,k\in\mathcal{K}.
  \label{eq:saturation}
\end{equation}
For $k\notin\mathcal{K}$, define
\begin{equation}
  M_k(U) := \sup_{x\in U}\,\max_{1\le i\le d_k}\lvert\sigma'(a_{k,i}(x))\rvert
  \label{eq:Mk}
\end{equation}
and set
\begin{equation}
  C(U) := C_W\cdot\delta^q\cdot\prod_{k\notin\mathcal{K}}M_k(U).
  \label{eq:CU}
\end{equation}
\end{definition}

\begin{remark}[Choice of $U$]
\label{rem:U}
For a periodic orbit $\gamma$, take $U=\mathrm{conv}(\gamma)$.
For a bounded trajectory $h(\cdot)\subset B$, any convex
$U\supseteq B$ suffices.
\end{remark}

\paragraph{Standing assumption.}
All bounds in Sections~\ref{sec:attenuation}--\ref{sec:refined}
apply to trajectories (or periodic orbits) that remain
in~$U$ for the entire time interval under consideration.
For a $T$-periodic orbit $\gamma$ the requirement is simply
$\gamma([0,T])\subset U$; see Remark~\ref{rem:U}.

\begin{remark}[When does saturation occur?]
\label{rem:saturation_practice}
For $\sigma=\tanh$, layer $k$ is $\delta$-saturated on $U$ whenever
$\min_{x\in U,\,i}|a_{k,i}(x)|\ge r$ with $\delta=\sech^2(r)$.
This holds when oscillation amplitude satisfies
$\sigma_{\min}(W_1)A\gg r$, or when a pre-activation scale
$s\gg 1$ is applied (Section~\ref{sec:numerics}),
provided the orbit avoids zero pre-activations.
\end{remark}

\section{Jacobian Attenuation under Activation Saturation}
\label{sec:attenuation}

This section establishes the mechanism underlying all subsequent
dynamical results: saturation of activation derivatives
attenuates the input Jacobian $Df_\theta(x)$.

\begin{lemma}[Gr\"{o}nwall \cite{Hartman1964,Chicone2006}]
\label{lem:gronwall}
Let $y:[t_0,T]\to[0,\infty)$ be absolutely continuous and
$\alpha:[t_0,T]\to[0,\infty)$ locally integrable.
If $y'(t)\le\alpha(t)y(t)$ a.e., then
\begin{equation}
  y(t)\le y(t_0)\exp\!\left(\int_{t_0}^t\alpha(\tau)\,d\tau\right),
  \qquad t\in[t_0,T].
  \label{eq:gronwall}
\end{equation}
\end{lemma}

\begin{lemma}[Trajectory sensitivity; standard, see e.g.\ {\cite{Hartman1964}}]
\label{lem:mvt}
Let $h,\tilde{h}$ solve~\eqref{eq:node} with both trajectories in a
convex set $U$ on $[t_0,t]$.  Then
\begin{equation}
  \norm{h(t)-\tilde{h}(t)}
  \le\exp\!\left(\int_{t_0}^t\sup_{\xi\in U}\norm{Df_\theta(\xi)}\,d\tau\right)
  \norm{h_0-\tilde{h}_0}.
  \label{eq:traj_sens}
\end{equation}
\end{lemma}

\begin{theorem}[Saturation-induced Jacobian attenuation]
\label{thm:main}
Let $U\subset\R^d$ be convex with $q$ layers $\delta$-saturated on $U$
(Definition~\ref{def:saturation}).
\begin{enumerate}[label=\textup{(\roman*)}]
\item For every $x\in U$,
  \begin{equation}
    \norm{Df_\theta(x)}\le C(U)
    = C_W\cdot\delta^q\cdot\prod_{k\notin\mathcal{K}}M_k(U).
    \label{eq:jac_attenuated}
  \end{equation}
\item If $h(\cdot),\tilde{h}(\cdot)$ remain in $U$ on $[t_0,t]$,
  \begin{equation}
    \norm{h(t)-\tilde{h}(t)}\le e^{C(U)(t-t_0)}\norm{h_0-\tilde{h}_0}.
    \label{eq:main_bound}
  \end{equation}
\item If $\Lambda_\sigma:=\sup_{x\in\R}|\sigma'(x)|\le 1$
  (e.g.\ $\sigma=\tanh$), then $C(U)\le C_W\delta^q$.
\end{enumerate}
\end{theorem}

\begin{proof}
\textit{(i)}
From~\eqref{eq:jac_bound} and~\eqref{eq:Dk_norm}:
\[
  \norm{Df_\theta(x)}
  \le C_W\prod_{k=1}^{L-1}\norm{D_k(x)}.
\]
We bound each factor separately for $x\in U$:
\begin{itemize}[noitemsep]
\item If $k\in\mathcal{K}$: by~\eqref{eq:saturation},
  $\norm{D_k(x)}=\max_i|\sigma'(a_{k,i}(x))|\le\delta$.
  There are $|\mathcal{K}|=q$ such factors, contributing $\delta^q$.
\item If $k\notin\mathcal{K}$: by definition~\eqref{eq:Mk},
  $\norm{D_k(x)}\le M_k(U)$.
\end{itemize}
Therefore
\[
  \norm{Df_\theta(x)}
  \le C_W\cdot\prod_{k\in\mathcal{K}}\delta\cdot\prod_{k\notin\mathcal{K}}M_k(U)
  = C_W\cdot\delta^q\cdot\prod_{k\notin\mathcal{K}}M_k(U) = C(U).
\]

\textit{(ii)}
Since both trajectories stay in $U$ on $[t_0,t]$,
part~(i) gives $\sup_{\xi\in U}\norm{Df_\theta(\xi)}\le C(U)$.
Substituting into Lemma~\ref{lem:mvt}:
\[
  \norm{h(t)-\tilde{h}(t)}
  \le\exp\!\left(\int_{t_0}^t C(U)\,d\tau\right)\norm{h_0-\tilde{h}_0}
  = e^{C(U)(t-t_0)}\norm{h_0-\tilde{h}_0}.
\]

\textit{(iii)}
For each $k\notin\mathcal{K}$ and $x\in U$:
$\max_i|\sigma'(a_{k,i}(x))|\le\sup_{z\in\R}|\sigma'(z)|=\Lambda_\sigma\le 1$,
so $M_k(U)\le 1$ and $\prod_{k\notin\mathcal{K}}M_k(U)\le 1$.
Thus $C(U)=C_W\delta^q\prod_{k\notin\mathcal{K}}M_k(U)\le C_W\delta^q$.
\end{proof}

\begin{remark}
For $\sigma=\tanh$ and $|a_{k,i}(x)|\ge r$ on $U$:
$\delta\le\sech^2(r)\le 4e^{-2r}$, so $\delta^q\le 4^q e^{-2rq}$.
\end{remark}

\begin{remark}[Tightness]\label{rem:tightness}
Bound~\eqref{eq:jac_attenuated} is sharp in $\delta$; the factor
$C_W$ is refined in Section~\ref{sec:refined}.
\end{remark}

\begin{corollary}[Activation comparison principle]
\label{cor:comparison}
Let $f^{(1)}_\theta$ and $f^{(2)}_\theta$ share weights $\{W_\ell,b_\ell\}$
but use activations $\sigma_1,\sigma_2$.  Let $h^{(1)}(\cdot)$ be a
trajectory of the $\sigma_1$-system.  If
\begin{equation}
  \max_i\lvert\sigma_1'(a^{(1)}_{k,i}(\tau))\rvert
  \le\max_i\lvert\sigma_2'(a^{(1)}_{k,i}(\tau))\rvert,
  \qquad\forall\,\tau\in[t_0,t],\;k=1,\ldots,L-1,
  \label{eq:comparison_hyp}
\end{equation}
then
\begin{equation}
  \int_{t_0}^t\norm{Df^{(1)}_\theta(h^{(1)}(\tau))}\,d\tau
  \le C_W\int_{t_0}^t\prod_{k=1}^{L-1}
    \max_i\lvert\sigma_2'(a^{(1)}_{k,i}(\tau))\rvert\,d\tau.
  \label{eq:comparison}
\end{equation}
In particular, for $\sigma_1=\tanh$, $\sigma_2=\mathrm{SiLU}$,
since $\mathrm{SiLU}'(x)\to 1$ as $x\to+\infty$, the
obstruction does not apply to non-saturating
activations~\cite{Ramachandran2017,Elfwing2018,LiLiuLiveraniZuazua2026}.
\end{corollary}

\begin{proof}
Apply~\eqref{eq:jac_bound} to $f^{(1)}_\theta$ at $h^{(1)}(\tau)$:
\[
  \norm{Df^{(1)}_\theta(h^{(1)}(\tau))}
  \le C_W\prod_{k=1}^{L-1}\max_i\lvert\sigma_1'(a^{(1)}_{k,i}(\tau))\rvert.
\]
By hypothesis~\eqref{eq:comparison_hyp}, replacing each factor
$\max_i|\sigma_1'(a^{(1)}_{k,i}(\tau))|$ by the larger
$\max_i|\sigma_2'(a^{(1)}_{k,i}(\tau))|$:
\[
  \norm{Df^{(1)}_\theta(h^{(1)}(\tau))}
  \le C_W\prod_{k=1}^{L-1}\max_i\lvert\sigma_2'(a^{(1)}_{k,i}(\tau))\rvert.
\]
Integrating over $[t_0,t]$ gives~\eqref{eq:comparison}.
\end{proof}

\section{The Floquet--Liouville Obstruction}
\label{sec:floquet}

This section contains the main results of the paper: the collapse
of the Floquet spectrum under activation saturation.

\subsection{Liouville formula}

\begin{lemma}[Liouville--Abel--Jacobi \cite{CoddingtonLevinson1955,Hartman1964,Chicone2006,GuckenheimerHolmes1983}]
\label{lem:liouville}
Let $f_\theta\in C^1(\R^d;\R^d)$ and let $h(t)=\Phi^\theta_t(h_0)$.
Set $\Psi(t):=D\Phi^\theta_t(h_0)\in\R^{d\times d}$.  Then $\Psi$ satisfies
the variational equation
\begin{equation}
  \dot{\Psi}(t) = Df_\theta(h(t))\,\Psi(t), \qquad \Psi(t_0)=I,
  \label{eq:variational}
\end{equation}
and
\begin{equation}
  \frac{d}{dt}\det\Psi(t)
  = \Tr\!\bigl(Df_\theta(h(t))\bigr)\cdot\det\Psi(t),
  \qquad \det\Psi(t_0)=1.
  \label{eq:liouville_ode}
\end{equation}
For any $T$-periodic orbit $\gamma$ with monodromy matrix
$M_\gamma:=D\Phi^\theta_T(h_0)=\Psi(T)$,
\begin{equation}
  \det M_\gamma
  = \exp\!\left(\int_0^T\Tr\!\bigl(Df_\theta(\gamma(t))\bigr)\,dt\right).
  \label{eq:liouville}
\end{equation}
\end{lemma}

\subsection{Trace bound}

\begin{lemma}[\cite{HornJohnson2013}]
\label{lem:trace}
For any $A\in\R^{d\times d}$, where $\norm{\cdot}=\norm{\cdot}_2$
denotes the operator (spectral) norm,
\begin{equation}
  \lvert\Tr(A)\rvert \le d\,\norm{A}.
  \label{eq:trace_bound}
\end{equation}
\end{lemma}

\begin{proof}
For each $i=1,\ldots,d$, let $e_i$ denote the $i$-th standard basis
vector.  Since $\norm{e_i}_2=1$,
\[
  |a_{ii}| = |e_i^\top A\, e_i| \le \norm{A}_2\,\norm{e_i}_2^2 = \norm{A}_2.
\]
Therefore
$\lvert\Tr(A)\rvert = \lvert\sum_{i=1}^d a_{ii}\rvert
\le \sum_{i=1}^d |a_{ii}| \le d\,\norm{A}_2$.
\end{proof}

\begin{lemma}[Rank--trace bound]
\label{lem:rank_trace}
For any $A\in\R^{d\times d}$ with $\mathrm{rank}(A)\le r$,
\begin{equation}
  \lvert\Tr(A)\rvert \le r\,\norm{A}.
  \label{eq:rank_trace}
\end{equation}
\end{lemma}

\begin{proof}
Since $\dim\ker(A)\ge d-r$, the eigenvalue $0$ has algebraic
multiplicity at least $d-r$, so $A$ has at most $r$ nonzero
eigenvalues $\lambda_1,\ldots,\lambda_r$ (counted with algebraic
multiplicity).
For each $\lambda_i$, let $v_i\in\mathbb{C}^d$ be a corresponding
unit eigenvector ($Av_i=\lambda_iv_i$, $\norm{v_i}_2=1$).
Since the spectral norm of a real matrix coincides with its
operator norm over $\mathbb{C}^d$,
we have
$|\lambda_i|=\norm{Av_i}_2\le\norm{A}\norm{v_i}_2=\norm{A}$.
Therefore
\[
  \lvert\Tr(A)\rvert = \left\lvert\sum_{i=1}^r\lambda_i\right\rvert
  \le \sum_{i=1}^r|\lambda_i| \le r\,\norm{A}. \qedhere
\]
\end{proof}

\subsection{Main theorems}

\begin{theorem}[Floquet--Liouville conservativity obstruction]
\label{thm:liouville}
Let $\gamma$ be a $T$-periodic orbit of~\eqref{eq:node} contained in
a convex set $U$ where $q$ layers are $\delta$-saturated, and let
$M_\gamma:=D\Phi^\theta_T(h_0)$ be the monodromy matrix.
\begin{enumerate}[label=\textup{(\roman*)}]
\item \begin{equation}
    \lvert\ln\det M_\gamma\rvert \le d\,C(U)\,T.
    \label{eq:floquet_bound}
  \end{equation}
\item Let $\mu_0=1$ denote the trivial Floquet
  multiplier (proved below) and $\mu_1,\ldots,\mu_{d-1}$ the
  remaining eigenvalues of $M_\gamma$ counted with algebraic
  multiplicity.  Then
  $\det M_\gamma
  = \prod_{i=0}^{d-1}\mu_i
  = \prod_{i=1}^{d-1}\mu_i$.
  By Lemma~\ref{lem:liouville}, $\det M_\gamma=\exp(\cdots)>0$,
  so the real logarithm $\ln\det M_\gamma$ is well defined.
  Therefore
  \begin{equation}
    \left\lvert\sum_{i=1}^{d-1}\ln|\mu_i|\right\rvert
    = \lvert\ln\det M_\gamma\rvert \le d\,C(U)\,T.
    \label{eq:floquet_product}
  \end{equation}
\item If additionally $\Lambda_\sigma\le 1$
  and $d\,C_W\delta^q T < \eta$ for some $\eta>0$,
  no $T$-periodic orbit $\gamma\subset U$ can achieve
  $|\det M_\gamma|\le e^{-\eta}$.
\end{enumerate}
\end{theorem}

\begin{proof}
\textit{(i)}
By Lemma~\ref{lem:liouville} and the triangle inequality:
\[
  |\ln\det M_\gamma|
  = \left|\int_0^T\Tr(Df_\theta(\gamma(t)))\,dt\right|
  \le \int_0^T\lvert\Tr(Df_\theta(\gamma(t)))\rvert\,dt.
\]
Applying Lemma~\ref{lem:trace} pointwise:
$|\Tr(Df_\theta(\gamma(t)))|\le d\,\norm{Df_\theta(\gamma(t))}$.
Since $\gamma(t)\in U$, Theorem~\ref{thm:main}(i) gives
$\norm{Df_\theta(\gamma(t))}\le C(U)$.  Integrating:
\[
  |\ln\det M_\gamma|
  \le \int_0^T d\,C(U)\,dt = d\,C(U)\,T.
\]

\textit{(ii)}
Set $v(t):=\Psi(t)\,f_\theta(h_0)$ and $w(t):=f_\theta(h(t))$.
Both solve~\eqref{eq:variational} with the same initial condition
$v(0)=w(0)=f_\theta(h_0)$, so $v\equiv w$ by uniqueness.
At $t=T$: $M_\gamma f_\theta(h_0)=v(T)=w(T)=f_\theta(h_0)\ne 0$
(since $\gamma$ is not an equilibrium),
giving the trivial multiplier $\mu_0=1$.
The remaining $\mu_1,\ldots,\mu_{d-1}$ live on the complementary
subspace, so $\det M_\gamma=\prod_{i=1}^{d-1}\mu_i$.
Bound~\eqref{eq:floquet_product} follows from~\eqref{eq:floquet_bound}.

\textit{(iii)}
Suppose $|\det M_\gamma|\le e^{-\eta}$.
By~\eqref{eq:liouville}, $\det M_\gamma>0$, so
$\ln\det M_\gamma\le -\eta<0$ and $|\ln\det M_\gamma|\ge\eta$.
Bound~\eqref{eq:floquet_bound} forces $d\,C(U)\,T\ge\eta$.
Since $C(U)\le C_W\delta^q$ when $M_k(U)\le 1$,
the condition $d\,C_W\delta^q T<\eta$
implies $d\,C(U)\,T<\eta$, a contradiction.
\end{proof}

\begin{theorem}[Individual Floquet multiplier bounds]
\label{thm:individual}
Let $\gamma\subset U$ be a $T$-periodic orbit of~\eqref{eq:node}
with monodromy matrix $M_\gamma$ and Floquet multipliers
$\mu_0=1,\mu_1,\ldots,\mu_{d-1}$.
Under the saturation condition of Definition~\ref{def:saturation},
\begin{equation}
  e^{-C(U)\,T} \;\le\; |\mu_i| \;\le\; e^{C(U)\,T},
  \qquad i=0,\ldots,d-1.
  \label{eq:indiv_bound}
\end{equation}
Equivalently, $|\ln|\mu_i||\le C(U)\,T$ for each $i$.
In particular, if $\Lambda_\sigma\le 1$
and $C_W\delta^q T<\eta$, no individual multiplier
satisfies $|\mu_i|\le e^{-\eta}$.
\end{theorem}

\begin{proof}
\textit{(Upper bound.)}
The fundamental matrix $\Psi(t)=D\Phi^\theta_t(h_0)$ satisfies the
variational equation~\eqref{eq:variational}.
Since $\gamma(t)\in U$ and $\norm{Df_\theta(\gamma(t))}\le C(U)$
by Theorem~\ref{thm:main}(i), the function
$v(t)=\Psi(t)w$ satisfies the linear equation
$\dot{v}=Df_\theta(\gamma(t))v$, so
$\frac{d}{dt}\norm{v}\le\norm{Df_\theta(\gamma(t))}\norm{v}\le C(U)\norm{v}$.
Lemma~\ref{lem:gronwall} then gives
$\norm{\Psi(t)w}\le e^{C(U)t}$$\norm{w}$
($\norm{w}=1$),
hence $\norm{\Psi(T)}\le e^{C(U)T}$.
For any eigenvalue $\mu_i$ with unit eigenvector $v_i$:
$|\mu_i|=\norm{M_\gamma v_i}=\norm{\Psi(T)v_i}\le\norm{\Psi(T)}\le e^{C(U)T}$.

\textit{(Lower bound.)}
Since $\gamma$ is $T$-periodic and $\gamma\subset U$, the backward
orbit $\gamma(-t)=\gamma(T-t)$ also lies entirely in $U$.
Consider the time-reversed system $\dot{u}=-f_\theta(u)$ with
$u(0)=h_0$; its solution is $u(t)=\Phi^\theta_{-t}(h_0)=\gamma(-t)$.
The variational equation for this system reads
\begin{equation}
  \dot{\Psi}_{-}(t) = -Df_\theta(\gamma(-t))\,\Psi_{-}(t),
  \qquad \Psi_{-}(0)=I,
  \label{eq:back_var}
\end{equation}
so $\Psi_{-}(T)=D\Phi^\theta_{-T}(h_0)$.
Since $\Phi^\theta_T(h_0)=h_0$ ($T$-periodicity), the chain rule
applied to $\Phi^\theta_{-T}\circ\Phi^\theta_T=\mathrm{id}$ gives
$D\Phi^\theta_{-T}(\underbrace{\Phi^\theta_T(h_0)}_{=\,h_0})
 \cdot D\Phi^\theta_T(h_0)=I$,
hence $\Psi_{-}(T)=\Psi(T)^{-1}=M_\gamma^{-1}$.
Since $\gamma(-t)\in U$ and $\norm{-Df_\theta(\gamma(-t))}=\norm{Df_\theta(\gamma(-t))}\le C(U)$,
the same Gronwall argument gives $\norm{\Psi_{-}(T)}\le e^{C(U)T}$,
i.e.\ $\norm{M_\gamma^{-1}}\le e^{C(U)T}$.
Since $\det M_\gamma>0$ by Lemma~\ref{lem:liouville},
all eigenvalues of $M_\gamma$ are nonzero, so $M_\gamma^{-1}$ exists
and $M_\gamma^{-1}v_i=(1/\mu_i)v_i$.
Therefore $1/|\mu_i|\le\norm{M_\gamma^{-1}}\le e^{C(U)T}$,
giving $|\mu_i|\ge e^{-C(U)T}$.

\textit{(Last claim.)}
Suppose $|\mu_i|\le e^{-\eta}$.
Then $-\ln|\mu_i|\ge\eta$, so $|\ln|\mu_i||\ge\eta$.
Bound~\eqref{eq:indiv_bound} forces $C(U)T\ge\eta$.
Since $C(U)\le C_W\delta^q$ (Theorem~\ref{thm:main}(iii)),
the condition $C_W\delta^q T<\eta$ implies $C(U)T<\eta$,
a contradiction.
\end{proof}

\begin{remark}[Connection to Floquet (Lyapunov) exponents]
\label{rem:lyapunov}
For a periodic orbit~$\gamma$ with period~$T$, the Floquet
exponents $\lambda_i:=\tfrac{1}{T}\ln|\mu_i|$ coincide with the
Lyapunov exponents of~$\gamma$.
Theorem~\ref{thm:individual} therefore implies
$|\lambda_i|\le C(U)$ for each~$i$, and when
$\Lambda_\sigma\le 1$ this sharpens to $|\lambda_i|\le C_W\delta^q$.
Deep saturation drives all exponents to zero,
limiting or suppressing chaotic sensitivity ($\lambda_i>0$) and
limiting the achievable contraction rate ($\lambda_i<0$ but $|\lambda_i|\to 0$).
In short, saturation collapses the dynamical spectrum
of periodic orbits within~$U$:
neither strong contraction nor strong instability can persist.
\end{remark}

\begin{corollary}[Contraction threshold]
\label{cor:threshold}
Let $\Lambda_\sigma\le 1$.
A periodic orbit $\gamma\subset U$ with period~$T$
can exhibit net contraction $|\det M_\gamma|\le e^{-\eta}$
only if the saturation level satisfies
\begin{equation}\label{eq:threshold}
  \delta
  \;\ge\;
  \Bigl(\frac{\eta}{d\,C_W\,T}\Bigr)^{1/q}.
\end{equation}
Moreover, every individual Floquet (Lyapunov) exponent
satisfies $|\lambda_i|\le C(U)$
(Remark~\ref{rem:lyapunov}).
\end{corollary}

\begin{proof}
By Theorem~\ref{thm:liouville}(i) and~(iii),
$d\,C_W\delta^q T<\eta$ implies
$|\ln\det M_\gamma|<\eta$, hence $e^{-\eta}<\det M_\gamma<e^{\eta}$.
Rearranging: $\delta<(\eta/(d\,C_W T))^{1/q}$.
\end{proof}

\begin{corollary}[Bottleneck Floquet--Liouville bound]
\label{cor:bottleneck}
Let the hypotheses of Theorem~\ref{thm:liouville} hold, and set
$r := \min_{1\le k\le L-1} d_k$
to be the minimum hidden-layer width (the \emph{bottleneck width}).
Then $\mathrm{rank}(Df_\theta(x))\le r$ for all $x\in\R^d$, and
\begin{equation}
  \lvert\ln\det M_\gamma\rvert \le r\,C(U)\,T.
  \label{eq:floquet_bottleneck}
\end{equation}
When $r<d$, this improves
Theorem~\ref{thm:liouville}\textup{(i)} by a factor $d/r$.
\end{corollary}

\begin{proof}
The Jacobian $Df_\theta(x)=W_L D_{L-1}W_{L-1}\cdots D_1 W_1$
is a product containing the factor $W_k\in\R^{d_k\times d_{k-1}}$ for
each $k$.
Submultiplicativity of rank gives
$\mathrm{rank}(Df_\theta(x))\le\min_k d_k=r$.
Replacing Lemma~\ref{lem:trace} by Lemma~\ref{lem:rank_trace} in the
proof of Theorem~\ref{thm:liouville}(i) yields
$\lvert\Tr(Df_\theta(\gamma(t)))\rvert
\le r\,C(U)$,
and integrating over $[0,T]$ gives~\eqref{eq:floquet_bottleneck}.
\end{proof}

\begin{remark}[Comparison of the two main theorems]
\label{rem:indiv_vs_det}
Theorem~\ref{thm:liouville}(i) bounds the logarithm of the
monodromy determinant
($|\ln\det M_\gamma|\le d\,C(U)T$),
controlling the \emph{product} of all multipliers.
Theorem~\ref{thm:individual} bounds \emph{each} multiplier
individually ($|\ln|\mu_i||\le C(U)T$).
Since $|\sum_i\ln|\mu_i||\le(d-1)\,C(U)T$,
Theorem~\ref{thm:individual} also implies
$|\ln\det M_\gamma|\le(d-1)\,C(U)T$,
recovering (and slightly sharpening) Theorem~\ref{thm:liouville}(i).
Theorem~\ref{thm:liouville} retains independent value:
its proof via the Liouville--Abel--Jacobi identity and the trace bound
generalises to the bottleneck setting
(Corollary~\ref{cor:bottleneck}), where the factor~$d$ sharpens to
the bottleneck width~$r$.
The individual bound is qualitatively stronger in a different sense:
for an asymptotically
stable orbit every non-trivial multiplier must lie in the interval
$(e^{-C(U)T},1)$, so the ``stability window'' per multiplier
shrinks to zero as $C(U)T\to 0$.
\end{remark}

\section{Numerical Verification}
\label{sec:numerics}

We illustrate the mechanism and bounds predicted by
the main theorems on the
\emph{Stuart--Landau oscillator}~\cite{Stuart1960,Landau1944},
chosen for its exact closed-form monodromy.
Additional numerical experiments are given in
Appendix~\ref{app:numerics}.

\paragraph{Setup.}
The Stuart--Landau system is
\begin{equation}
  \dot{x} = x - y - x(x^2+y^2),\qquad
  \dot{y} = x + y - y(x^2+y^2),
  \label{eq:sl}
\end{equation}
which possesses an asymptotically stable limit cycle $\gamma$ on the
unit circle with period $T=2\pi$.
Along $\gamma$ the Jacobian trace is identically
$\Tr(Df_{\mathrm{SL}}(\gamma(t)))=-2$,
giving $\ln\det M_\gamma = -4\pi$,
i.e.\ $\det M_\gamma = e^{-4\pi}\approx 3.487\times 10^{-6}$.

We fit a two-layer $\tanh$-MLP Neural ODE
$\dot{h}=f_\theta(h;\,s)$,
with architecture
$f(h;\,s)=W_2\tanh(s(W_1 h+b_1))+b_2$,
to the Stuart--Landau vector field by minimising the mean-squared
residual on sample points in the annulus $0.1\le\norm{h}\le 2$.
The pre-activation scale $s>0$ controls saturation depth: large $s$
drives the hidden units into the flat region of $\tanh$.
Illustrations~A and~C use 32 and 256 hidden units respectively; all runs are seeded for reproducibility (see \emph{Code availability} at the end of the paper).

\paragraph{Illustration~A: Jacobian attenuation (Theorem~\ref{thm:main}).}
For each value of $s$ we compute the spectral norm
$\norm{Df_\theta(h_0;\,s)}_2$ at a fixed evaluation point
$h_0=(0.8,\,0.4)$ and compare against the theoretical bound
$C_W(s)\cdot\delta(h_0,s)
 = s\,\norm{W_1}_2\,\norm{W_2}_2\cdot\max_i\sech^2(s\,a_i)$
(Theorem~\ref{thm:main}, single hidden layer, $q=1$).

\begin{figure}[ht]
  \centering
  \includegraphics[width=0.80\linewidth]{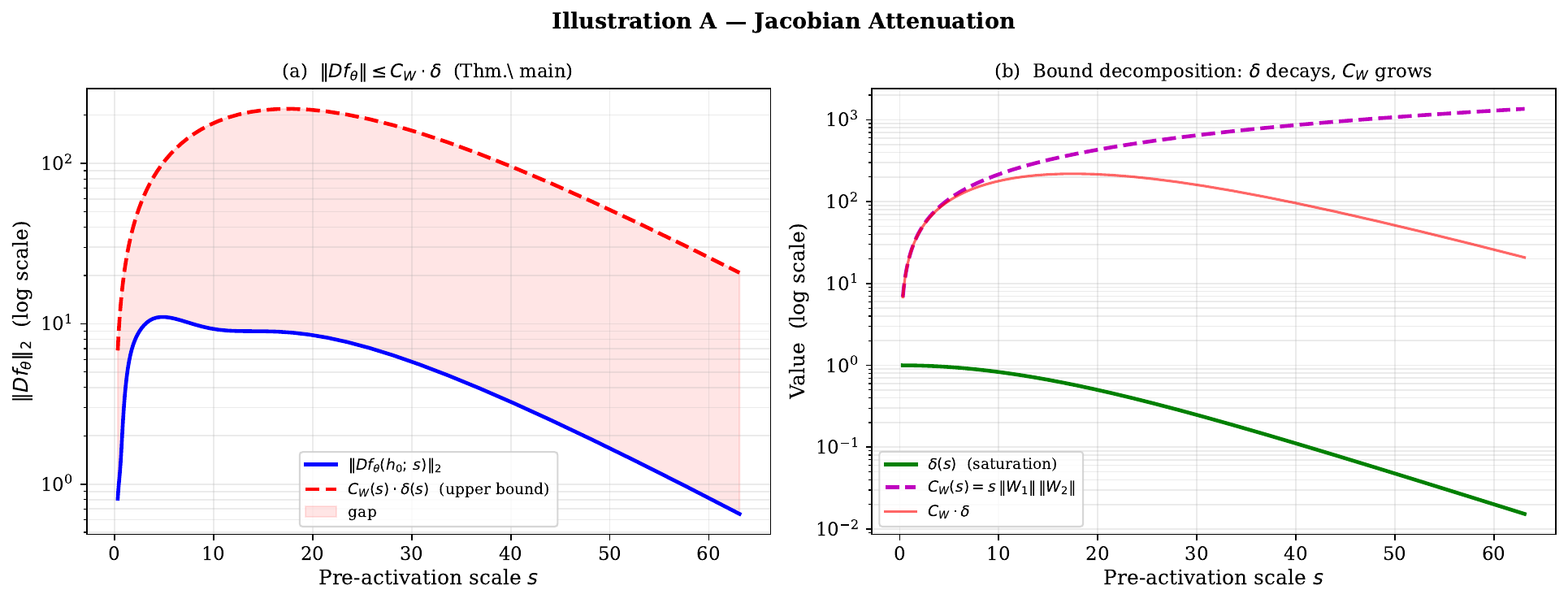}
  \caption{%
    \textbf{Jacobian attenuation (Theorem~\ref{thm:main}).}
    \emph{(a)}~Spectral norm $\norm{Df_\theta}_2$ (solid blue) and
    upper bound $C_W(s)\cdot\delta(s)$ (dashed red) as functions
    of the pre-activation scale~$s$.
    The shaded region is the gap between the bound and the true norm;
    the bound is satisfied for all~$s$.
    \emph{(b)}~Decomposition of the bound: the weight factor
    $C_W(s)=s\,\norm{W_1}_2\,\norm{W_2}_2$ grows linearly in~$s$,
    but the saturation factor $\delta(s)$ decays exponentially,
    so their product still vanishes.}
  \label{fig:exp_A}
\end{figure}

Figure~\ref{fig:exp_A} confirms that both the Jacobian norm
and the bound decay to zero as $s\to\infty$.

\paragraph{Illustration~C: Phase portraits and Floquet spectrum collapse
(diagnostic evaluation of the bounds underlying
Theorems~\ref{thm:liouville}--\ref{thm:individual}).}
We illustrate the dynamical consequence of Jacobian attenuation
in a controlled setting where the saturation hypothesis holds
by construction.
We impose a non-trainable bias offset $c=2.5$ on every
pre-activation during training and constrain
$\lVert W_{1,j\cdot}\rVert_2 \le c - 0.5 = 2.0$,
ensuring all pre-activations remain strictly positive on the orbit
(see Appendix~\ref{app:numerics} for the complete protocol).
For each value of~$s$ we also evaluate the Floquet--Liouville integral
$\ln\det M_\gamma(s)
=\int_0^T \Tr\!\bigl(Df_\theta(\gamma(t);\,s)\bigr)\,dt$
numerically along the orbit.

\begin{figure}[ht]
  \centering
  \includegraphics[width=\linewidth]{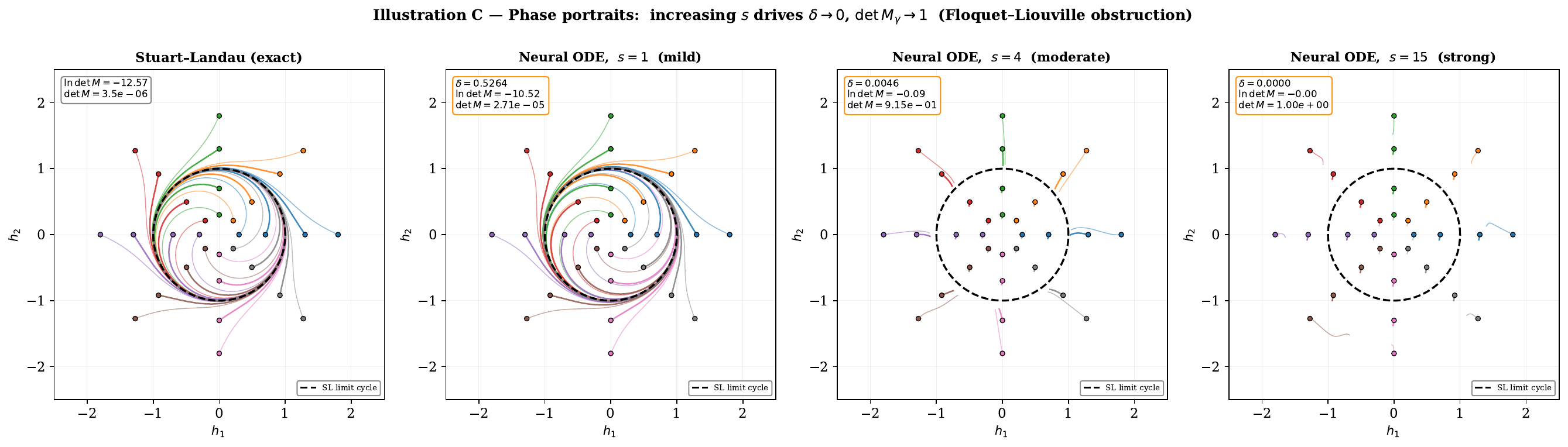}
  \caption{%
    \textbf{Phase portraits: exact Stuart--Landau vs.\ Neural ODE
    at increasing saturation.}
    \emph{Left:} exact Stuart--Landau flow
    ($\ln\det M=-4\pi$, strongly attracting limit cycle).
    \emph{Panels~2--4:} bias-shifted 256-unit MLP at $s=1,4,15$.
    Dashed circle: Stuart--Landau limit cycle ($r=1$).
    Annotations show $\delta$, $\ln\det M$, and $\det M$.
    At $s=1$ ($\delta\approx 0.53$, $\det M\approx 10^{-5}$) the MLP
    reproduces dissipative spirals; at $s=4$ ($\delta\approx 0.005$,
    $\det M\approx 0.92$) convergence weakens visibly;
    at $s=15$ ($\delta\approx 0$, $\det M=1$) trajectories are
    nearly rectilinear and the limit cycle no longer attracts.}
  \label{fig:exp_C}
\end{figure}

As $\delta\to 0$ (Figure~\ref{fig:exp_C}),
the obstruction forces $\det M\to 1$:
no volume-contracting periodic orbit survives in the saturated regime.

\section{Refined Jacobian Bounds via Saturation-Weighted Spectral Factorisation}
\label{sec:refined}

For activations with $\sigma'(a)>0$
(e.g.\ $\tanh$, sigmoid, softplus),
each $D_k(x)$ is positive definite
and admits a square root $D_k^{1/2}$.
Redistributing $D_k^{1/2}$ to adjacent weight matrices
before applying submultiplicativity yields a tighter bound
$\widetilde{C}(U)\le C(U)$.

\begin{theorem}[Refined bound --- general depth]\label{thm:multi}
Let $f_\theta$ be an $L$-layer MLP~\eqref{eq:mlp}
with $\sigma'>0$, and assume $q$ layers are $\delta$-saturated
on~$U$ (Definition~\textup{\ref{def:saturation}}).
\begin{enumerate}[label=\textup{(\roman*)}]
\item \emph{Pointwise refined bound.}
For every $x\in U$,
\begin{equation}\label{eq:multi}
  \norm{Df_\theta(x)}
  \;\le\;
  \norm{W_L\,D_{L-1}^{1/2}}
  \;\cdot\;
  \prod_{k=2}^{L-1}
    \norm{D_k^{1/2}\,W_k\,D_{k-1}^{1/2}}
  \;\cdot\;
  \norm{D_1^{1/2}\,W_1},
\end{equation}
where all $D_k=D_k(x)$.

\item \emph{Uniform refined bound.}
Define
\begin{equation}\label{eq:Ctilde}
  \widetilde{C}(U)
  \;:=\;
  \sup_{x\in U}\;
  \norm{W_L\,D_{L-1}(x)^{1/2}}
  \;\cdot\;
  \prod_{k=2}^{L-1}
    \norm{D_k(x)^{1/2}\,W_k\,D_{k-1}(x)^{1/2}}
  \;\cdot\;
  \norm{D_1(x)^{1/2}\,W_1}.
\end{equation}
Then
\begin{equation}\label{eq:chain}
  \norm{Df_\theta(x)}
  \;\le\;
  \widetilde{C}(U)
  \;\le\;
  C(U).
\end{equation}
The improvement ratio $\rho:=C(U)/\widetilde{C}(U)\ge 1$
is governed entirely by the $q$ saturated layers.
\end{enumerate}
\end{theorem}

\begin{proof}
\textit{(i)}
Factor each $D_k=D_k^{1/2}\,D_k^{1/2}$ in~\eqref{eq:jacobian}
and regroup adjacent terms:
\begin{align*}
  Df_\theta(x)
  &= W_L\,D_{L-1}\,W_{L-1}\;\cdots\;D_1\,W_1 \\
  &= \underbrace{W_L\,D_{L-1}^{1/2}}_{\text{left edge}}\;\cdot\;
     \underbrace{D_{L-1}^{1/2}\,W_{L-1}\,D_{L-2}^{1/2}}_{\text{interior}}\;\cdot\;
     \;\cdots\;\cdot\;
     \underbrace{D_2^{1/2}\,W_2\,D_1^{1/2}}_{\text{interior}}\;\cdot\;
     \underbrace{D_1^{1/2}\,W_1}_{\text{right edge}}.
\end{align*}
Apply submultiplicativity to obtain~\eqref{eq:multi}.

\textit{(ii)}
The first inequality in~\eqref{eq:chain} follows from~(i).
For the second, apply submultiplicativity to each grouped factor
to separate weights and activations:
each $\norm{D_k}$ contributes $\norm{D_k}^{1/2}$ from two adjacent factors,
yielding $\norm{D_k}$.
Taking the supremum over~$U$ gives
$\widetilde{C}(U)\le C_W\prod_{k=1}^{L-1}M_k(U)$.
Separating the $q$ saturated layers:
$\prod_{k\in\mathcal{K}}M_k(U)\le\delta^q$,
hence $\widetilde{C}(U)\le C(U)$.
\end{proof}

\begin{proposition}[Attainability]\label{prop:sharp}
For any positive diagonal $D\in\R^{n\times n}$
there exist $W_1\in\R^{n\times d}$, $W_2\in\R^{d\times n}$
such that equality holds in the single-layer version
of~\eqref{eq:multi}:
$\norm{W_2\,D\,W_1}
=\norm{W_2\,D^{1/2}}\cdot\norm{D^{1/2}\,W_1}$.
Thus the refined bound is sharp.
\end{proposition}

\begin{proof}
Given any unit vector $v\in\R^n$, choose
$W_1=D^{-1/2}\,v\,e_1^\top$ and
$W_2=e_1\,v^\top\,D^{-1/2}$.
Then $B=D^{1/2}W_1=v\,e_1^\top$
and $A=W_2 D^{1/2}=e_1\,v^\top$,
so $\norm{AB}=1=\norm{A}\,\norm{B}$.
\end{proof}

The improvement ratio $\rho:=C(U)/\widetilde{C}(U)$
is amplified to $e^{(\rho-1)\widetilde{C}(U)T}$
at the flow level via Gr\"onwall's inequality
(Corollary~\ref{cor:exp_amp} in Appendix~\ref{app:auxiliary}),
and leads to a refined Floquet obstruction
$e^{-d\widetilde{C}(U)T}\le\det M_\gamma\le e^{d\widetilde{C}(U)T}$
(Corollary~\ref{cor:floquet_refined}).

\section{Discussion}
\label{sec:discussion}

\subsection*{The training--inference gap}

The obstruction is invisible to training diagnostics:
$\mathcal{L}(\theta)$ can be small while the network cannot sustain
the target dynamics under saturation
(Illustration~C: MSE~$\approx 0.001$ at $s=1$,
$\det M_\gamma=1$ at $s=15$).
Batch normalisation~\cite{IoffeS2015} does not resolve this gap:
BN keeps normalised pre-activations near zero during training,
mitigating vanishing gradients in $\nabla_\theta\mathcal{L}$,
but at inference the relevant quantity is
$|\sigma'(a_{k,i}(h(t)))|$ along the target trajectory ---
a different evaluation point that BN does not control.

\subsection*{Application to the Morris--Lecar model}

{\tolerance=800\emergencystretch=1em
The Morris--Lecar (ML) model~\cite{MorrisLecar1981}
exhibits Hopf, SNLC, and homoclinic bifurcation
regimes~\cite{GuckenheimerHolmes1983}.
Empirically, a $\tanh$-NODE fails in all three while
$\mathrm{SiLU}$ succeeds~\cite{MatzakosSfyrakis2026}.
Our bounds provide a theoretical explanation:
ML voltage amplitudes ($30$--$40\,\mathrm{mV}$) can drive
pre-activations into the flat region of $\tanh$,
yielding deep saturation ($\delta\ll 1$).
Theorem~\ref{thm:main}(i) then gives
$\norm{Df_\theta}\le C_W\delta^q\ll\lambda_0$,
too small to sustain the required instability
for homoclinic orbits or strong enough contraction
for asymptotically stable limit cycles.
SiLU avoids positive-side saturation
($\mathrm{SiLU}'(x)\to 1$ as $x\to+\infty$),
so the obstruction disappears.\par}

\begin{remark}[Homoclinic orbits]\label{rem:homoclinic}
For a homoclinic orbit with saddle $x^*$,
$T\to\infty$ makes the Floquet bound vacuous.
A weaker necessary condition is
$\norm{Df_\theta(x^*)}\ge\lambda_0$,
where $\lambda_0:=\min_i|\mathrm{Re}(\lambda_i)|>0$
for the target saddle linearisation.
Theorem~\ref{thm:main}(i) gives $C_W\delta^q<\lambda_0$
under strong saturation, contradicting this condition.
A full homoclinic theory remains open.
\end{remark}

\subsection*{Limitations, scope, and extensions}

The bound is sharp in~$\delta$ (Remark~\ref{rem:tightness});
the refined bound of Section~\ref{sec:refined}
closes the $C_W$ gap to ${\sim}5\%$
under strong saturation (Appendix~\ref{app:numerics}).
Key limitations:
(i)~$T<\infty$ excludes homoclinic orbits (Remark~\ref{rem:homoclinic});
(ii)~zero crossings of pre-activations yield $\delta=1$;
(iii)~the refined factorisation requires $\sigma'>0$.
The obstruction requires both autonomy and saturation:
non-saturating activations (SiLU, ReLU, ELU) avoid it
(Corollary~\ref{cor:comparison}),
SA-NODEs~\cite{LiLiuLiveraniZuazua2026}
admit universal approximation results that bypass
the autonomous saturation obstruction studied here,
and augmented NODEs~\cite{Dupont2019AugmentedNODE}
loosen but do not eliminate the bound.
The Jacobian attenuation mechanism
(Theorem~\ref{thm:main}) extends beyond ODEs:
discrete residual networks heuristically have
one-step Jacobian $\approx I$ under saturation, and in FFJORD~\cite{Grathwohl2019FFJORD}
saturation drives $\Tr(Df_\theta)\to 0$.

\medskip\noindent\textbf{Relation to structure-preserving numerical methods.}
The obstruction identified in this paper is conceptually distinct from the limitations
studied in the theory of geometric and symplectic
integrators~\cite{Leimkuhler2004,HairerLubichWanner2006}.
Structure-preserving discretisations are designed to maintain geometric invariants or
qualitative properties of a \emph{given, known} dynamical system under numerical
time-stepping.
The present obstruction, by contrast, is a \emph{representability constraint}: it
concerns the class of learned autonomous vector fields $\{f_\theta\}$ and restricts what
dynamics the parametrised model can encode, independently of how the resulting ODE is
subsequently integrated.
A Neural ODE may fail to reproduce strongly stable periodic dynamics even if the
numerical solver integrates the learned vector field exactly, because the learned
continuous vector field $f_\theta$ itself cannot sustain the required Floquet structure.
The two analyses are therefore complementary: geometric integrators address how
faithfully a \emph{known} dynamics is propagated in time; the present work addresses how
faithfully a \emph{learned} vector field can represent the structural properties of the
target system.

\section{Conclusion}
\label{sec:conclusion}

This paper identifies and quantifies a structural inference-time limitation of autonomous
Neural ODEs with saturating activations.
The core mechanism is Jacobian attenuation (Theorem~\ref{thm:main}): if $q$ hidden
layers are $\delta$-saturated on a convex region~$U$, the operator norm of the learned
Jacobian satisfies
\[
  \norm{Df_\theta(x)} \le C(U)
  := C_W\cdot\delta^q\cdot\prod_{k\notin\mathcal{K}}M_k(U)
\]
throughout~$U$ (which reduces to $C_W\delta^q$ for $\sigma=\tanh$ since $\Lambda_\sigma\le 1$),
a quantity that vanishes as $\delta \to 0$.
Via the Liouville--Abel--Jacobi identity, this translates into a dynamical constraint
(Theorems~\ref{thm:liouville} and~\ref{thm:individual}): every Floquet exponent $\lambda_i$ is confined to the symmetric
interval $[-C(U),\,C(U)]$, while every Floquet multiplier $\mu_i$
lies in the multiplicative band $[e^{-C(U)T},\,e^{C(U)T}]$.

The dynamical consequence is substantial.
Under deep saturation, the Floquet spectrum collapses toward the neutral point: no
periodic orbit in the saturated region can exhibit strong orbital attraction or strong
instability.
This constitutes a \emph{structural bottleneck in the dynamical expressivity} of the
model --- not in the approximation-theoretic sense, but in the sense that the achievable
contraction and expansion rates of periodic flows are quantitatively bounded by the
saturation depth.

Crucially, this is not a training pathology.
The obstruction is a property of the \emph{learned vector field at inference time},
independent of the training loss or optimisation quality.
A model that achieves near-zero MSE on trajectory data may nonetheless produce dynamics
that are neutrally stable or structurally incorrect, because the saturation constraint
prevents the Jacobian from attaining the values required for strong orbital stability.
This training--inference gap is invisible to standard diagnostics and persists even when
batch normalisation is employed during training.

It is important to be precise about what this result does not claim.
The obstruction does not exclude the existence of periodic orbits within the saturated
region --- it restricts their dynamical character.
It does not contradict universal approximation theorems, which concern pointwise function
approximation and not dynamical behaviour.
The limitation is localised: it applies only within the saturated region~$U$, and
trajectories that traverse unsaturated parts of the state space are not governed by
these bounds.
For non-saturating activations such as SiLU, $\delta$ is bounded away from zero and
the obstruction disappears (Corollary~\ref{cor:comparison}).

The analysis points to a concrete direction for remediation.
Since the obstruction is architectural in nature, it is not removable in general by
optimisation alone unless training explicitly prevents saturation or changes the
activation regime:
it points toward either the adoption of non-saturating activations, or a training
procedure that embeds dynamical structure as an inductive bias --- for instance by penalising violations of the
trace condition for orbital stability directly during training.
We identify the design of such Floquet-aware training objectives as a natural
direction for future work.

\subsection*{Code availability}

All scripts reproducing the figures are available at
\href{https://github.com/nikmatz/Activation-Saturation-and-Floquet-Spectrum-Collapse-in-Neural-ODEs}{\texttt{https://github.com/nikmatz/Activation\discretionary{-}{}{-}Saturation\discretionary{-}{}{-}and\discretionary{-}{}{-}Floquet\discretionary{-}{}{-}Spectrum\discretionary{-}{}{-}Collapse\discretionary{-}{}{-}in\discretionary{-}{}{-}Neural\discretionary{-}{}{-}ODEs}}.

\subsection*{Acknowledgements}

The author thanks Professor Enrique Zuazua for his hospitality during a research
visit at the Chair of Dynamics, Control, Machine Learning and Numerics,
Friedrich-Alexander-Universit\"at Erlangen--N\"urnberg (FAU), and for
helpful discussions.
The author also thanks Dr.\ Lorenzo Liverani for helpful discussions
and feedback.
This work was carried out during a sabbatical leave from the
School of Pedagogical \&\ Technological Education (ASPETE), Greece.

\appendix

\section{Auxiliary Results}
\label{app:auxiliary}

This appendix collects auxiliary corollaries of
Theorem~\ref{thm:main} that are not needed for the
main narrative but may be of independent interest.

\begin{proposition}[Global Lipschitz bound; classical, cf.\ {\cite{Chen2018NeuralODE}}]
\label{prop:global}
If $\Lambda_\sigma:=\sup_{x\in\R}|\sigma'(x)|<\infty$, then
$f_\theta$ is globally Lipschitz with constant
$L_f \le C_W\Lambda_\sigma^{L-1}$,
and solutions of~\eqref{eq:node} satisfy
\begin{equation}
  \norm{h(t)-\tilde{h}(t)}\le e^{L_f(t-t_0)}\norm{h_0-\tilde{h}_0}.
  \label{eq:global_sens}
\end{equation}
\end{proposition}

\begin{corollary}[Simplified bound for bounded activations]
\label{cor:simplified}
If $\Lambda_\sigma<\infty$, then for all $x\in U$:
$\norm{Df_\theta(x)} \le C_W\,\Lambda_\sigma^{L-1-q}\,\delta^q$.
\end{corollary}

\begin{corollary}[Localized Lipschitz constant]
\label{cor:lipschitz}
Under the hypotheses of Theorem~\ref{thm:main}:
$\mathrm{Lip}(f_\theta;\allowbreak U)\le C(U)$.
\end{corollary}

\begin{proof}
Let $x,y\in U$.  Since $U$ is convex, the segment
$\ell(s):=y+s(x-y)$, $s\in[0,1]$, lies in $U$.  By the
fundamental theorem of calculus:
\[
  \norm{f_\theta(x)-f_\theta(y)}
  = \left\lVert\int_0^1 Df_\theta(\ell(s))\,(x-y)\,ds\right\rVert
  \le C(U)\norm{x-y},
\]
where the last step uses~\eqref{eq:jac_attenuated}.
\end{proof}

\begin{corollary}[Stiffness proxy]
\label{cor:stiffness}
Under the saturation condition~\eqref{eq:saturation}, if the trajectory
$h(\cdot)$ remains in $U$, then
$\Gamma(t):=\int_{t_0}^t\norm{Df_\theta(h(\tau))}\,d\tau
\le C(U)\,(t-t_0)$.
Since step-size restrictions for explicit methods depend on
$\norm{Df_\theta(h(\tau))}$~\cite{HairerNorsettWanner1993},
saturation reduces integration cost.
\end{corollary}

\begin{proposition}[Non-saturating activations admit stable limit cycles]
\label{prop:existence}
Let $g:\R^d\to\R^d$ be $C^1$ with a hyperbolic asymptotically stable
limit cycle $\gamma_0$ of period $T_0$, and let $K$ be a compact
neighbourhood of $\gamma_0$.
Let $\sigma\in C^2(\R)$ be non-polynomial.
Then there exists $\varepsilon^*>0$ such that for every
$\varepsilon\in(0,\varepsilon^*)$ there exist $L$, widths $\{d_\ell\}$,
and weights $\theta$ with
\begin{enumerate}[label=\textup{(\roman*)}]
\item $\norm{f_\theta-g}_{C^1(K)}<\varepsilon$,
\item the Neural ODE $\dot{h}=f_\theta(h)$ has an asymptotically
  stable limit cycle $\gamma_\theta$ with
  $d_H(\gamma_\theta,\gamma_0)<C\varepsilon$,
  where $C>0$ depends only on $(\gamma_0,K)$.
\end{enumerate}
\end{proposition}

\begin{proof}
\textit{(i)}
Since $\sigma$ is non-polynomial and of class $C^2$, multilayer
feedforward networks are dense in
$C^1(K;\R^d)$ in the $C^1$ norm on compact sets
\cite{Hornik1990}.

\textit{(ii)}
The Poincar\'e return map $P_g$ has a hyperbolic fixed point at
$p_0=\gamma_0\cap\Sigma$ ($\Sigma$ a transversal).
By the implicit function theorem and continuous dependence of the
Poincar\'e map on the vector field in the $C^1$
topology~\cite{Chicone2006},
any $f_\theta$ with $\norm{f_\theta-g}_{C^1(K)}<\varepsilon^*$
has an asymptotically stable limit cycle $\gamma_\theta$ with
$d_H(\gamma_\theta,\gamma_0)<C\varepsilon$.
\end{proof}

\medskip

\noindent\textit{Refined bound details.}
The following results provide additional details for
Section~\ref{sec:refined}.

\begin{theorem}[Refined bound --- single layer]\label{thm:single}
Let $f_\theta(x)=W_2\,\sigma(W_1 x+b_1)+b_2$
with $\sigma'>0$.
For every $x$,
\begin{equation}\label{eq:single}
  \norm{Df_\theta(x)}
  \;\le\;
  \norm{W_2\,D(x)^{1/2}}
  \;\cdot\;
  \norm{D(x)^{1/2}\,W_1}
  \;\;\le\;\;
  \norm{W_2}\;\delta(x)\;\norm{W_1},
\end{equation}
so the refined bound is never worse than the original.
\end{theorem}

\begin{proof}
Write $Df_\theta(x) = (W_2\,D(x)^{1/2})(D(x)^{1/2}\,W_1)$
and apply submultiplicativity.
For the second inequality, apply submultiplicativity once more
to each factor and use $\norm{D(x)^{1/2}}=\delta(x)^{1/2}$.
\end{proof}

\begin{remark}[When is the improvement strict?]\label{rmk:when}
Equality in~\eqref{eq:single} holds if and only if
the top right singular vector of $W_2$
coincides with the coordinate direction of $\max_j\sech(a_j)$,
and similarly for the top left singular vector of $W_1$.
Generically, this alignment does not hold,
and the inequality is strict.
The improvement is greatest when the diagonal entries of~$D(x)$
are heterogeneous:
some neurons near zero pre-activation ($\sigma'\approx\sigma'_{\max}$)
while others are deep in saturation ($\sigma'\approx 0$).
\end{remark}

\begin{definition}[Doubly-weighted spectral norm]\label{def:doubly}
For $A\in\R^{m\times n}$ and
positive-definite diagonal matrices $\Delta_L,\Delta_R$,
define
$\norm{A}_{\Delta_L,\Delta_R}
:=\norm{\Delta_L^{1/2}\,A\,\Delta_R^{1/2}}$.
The interior factors in~\eqref{eq:multi} are
$\norm{W_k}_{D_k,\,D_{k-1}}$.
\end{definition}

\begin{theorem}[Refined flow Lipschitz bound]\label{thm:flow}
For initial conditions $h_0,h_0'\in U$ and as long as both
trajectories remain in~$U$,
\begin{equation}\label{eq:flow}
  \norm{\Phi^\theta_T(h_0)-\Phi^\theta_T(h_0')}
  \;\le\;
  e^{\widetilde{C}(U)\,T}\,\norm{h_0-h_0'}
  \;\le\;
  e^{C(U)\,T}\,\norm{h_0-h_0'}.
\end{equation}
\end{theorem}

\begin{corollary}[Exponential amplification of the improvement]
\label{cor:exp_amp}
Define the improvement ratio
$\rho:=C(U)/\widetilde{C}(U)\ge 1$.
The ratio of the original to the refined flow bound satisfies
\begin{equation}\label{eq:exp_amp}
  \frac{e^{C(U)\,T}}{e^{\widetilde{C}(U)\,T}}
  \;=\;
  e^{(\rho-1)\,\widetilde{C}(U)\,T}
  \;\xrightarrow{T\to\infty}\;\infty
  \quad\text{whenever $\rho>1$.}
\end{equation}
\end{corollary}

\begin{corollary}[Refined Floquet obstruction]\label{cor:floquet_refined}
Under the hypotheses of Theorem~\ref{thm:liouville},
\begin{equation}\label{eq:floquet_refined}
  e^{-d\,\widetilde{C}(U)\,T}
  \;\le\;
  \det M_\gamma
  \;\le\;
  e^{+d\,\widetilde{C}(U)\,T},
\end{equation}
where $\widetilde{C}(U)\le C(U)$ by Theorem~\textup{\ref{thm:multi}(ii)}.
\end{corollary}

\section{Additional Numerical Experiments}
\label{app:numerics}

\paragraph{Illustration~B: Floquet--Liouville obstruction
(Theorem~\ref{thm:liouville}).}
The obstruction requires pre-activations bounded away from zero
(Figure~\ref{fig:exp_B}).
With small biases, pre-activations cross zero on the orbit,
keeping $\sech^2(0)=1$ and $C(U)$ non-decaying.
We shift $b_1\mapsto b_1+c\cdot\mathbf{1}$
with $c\ge c_{\min}:=\max_j\norm{W_{1j\cdot}}_2$.
Then $a_j(t)\ge c-\norm{W_{1j\cdot}}_2>0$
for all~$t$, eliminating zero crossings.

\begin{figure}[ht]
  \centering
  \includegraphics[width=0.72\linewidth]{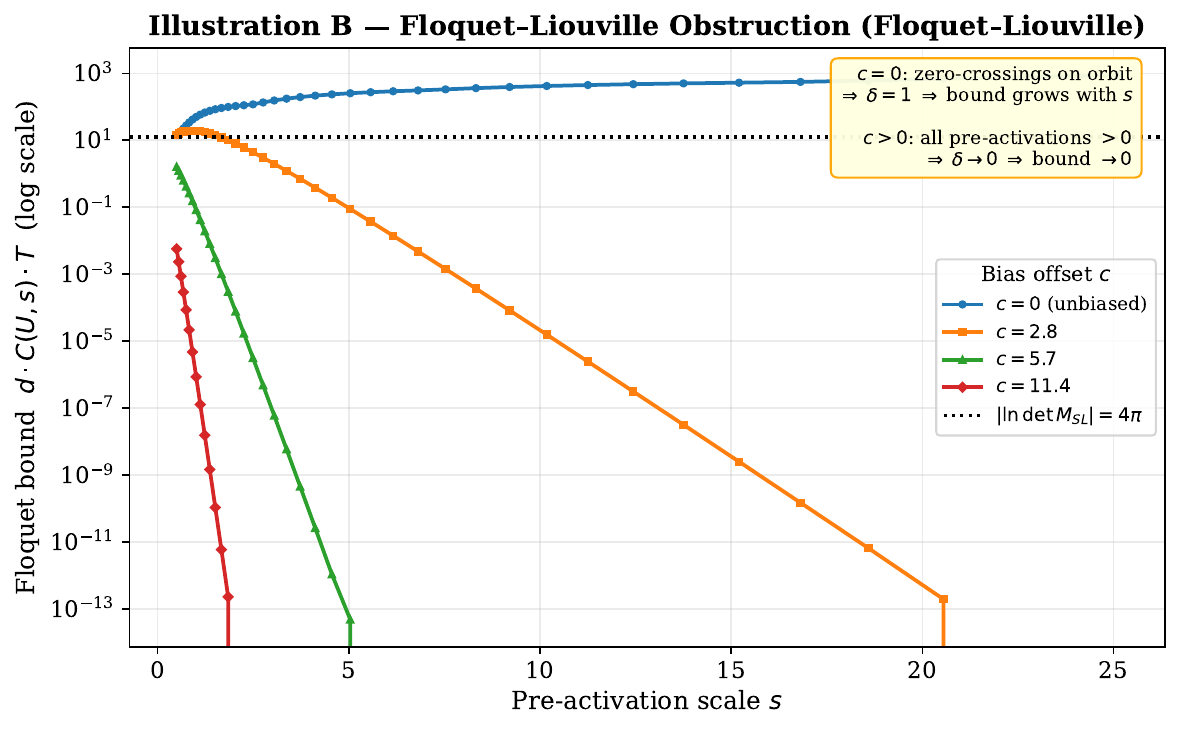}
  \caption{%
    \textbf{Floquet--Liouville obstruction (Theorem~\ref{thm:liouville}).}
    The obstruction bound $d\,C(U;\,s)\,T$ on $|\ln\det M_\gamma|$ as
    a function of pre-activation scale~$s$, for four bias-offset
    values $c\in\{0,\,1.5c_{\min},\,3c_{\min},\,6c_{\min}\}$
    where $c_{\min}=\max_j\norm{W_{1j\cdot}}_2$.
    Dotted line: the exact Stuart--Landau value
    $|\ln\det M_{\mathrm{SL}}|=4\pi$.
    The unbiased case ($c=0$, top curve) \emph{grows} with $s$ because
    zero crossings keep $\delta=1$.
    Each positive offset eliminates zero crossings; with $c=6c_{\min}$
    the bound falls below $10^{-2}$ at moderate~$s$.}
  \label{fig:exp_B}
\end{figure}

\paragraph{Illustration~C: Full protocol.}
\label{par:exp_C_protocol}
The bias-shifted experiment (Figure~\ref{fig:exp_C}) imposes
three constraints:
(1)~A non-trainable bias offset $c=2.5$ is added to every
pre-activation during training, so the network learns to
approximate Stuart--Landau \emph{in the presence of} the offset.
(2)~The bias vector $b_1$ is frozen at zero throughout training,
preventing the optimiser from undoing the shift.
(3)~The row norms of $W_1$ are projected after each step to satisfy
$\lVert W_{1,j\cdot}\rVert_2 \le c - 0.5 = 2.0$.
After training the offset is absorbed: $b_1\leftarrow b_1+c$.
By the triangle inequality, for every $h$ on the unit circle and
every hidden unit~$j$,
$a_j(h) = W_{1,j\cdot}h + b_{1,j}
\ge c - \lVert W_{1,j\cdot}\rVert_2
\ge 0.5 > 0$,
so all pre-activations are strictly positive on the orbit.

\paragraph{Illustration~D: Liouville--Abel--Jacobi identity.}
Table~\ref{tab:exp_D} verifies the identity
$\ln\det M_\gamma = \int_0^T \Tr\bigl(Df_{\mathrm{SL}}(\gamma(t))\bigr)\,dt$
on the Stuart--Landau system.

\begin{table}[ht]
  \centering
  \caption{%
    \textbf{LAJ identity for the Stuart--Landau oscillator.}
    Numerical integration vs.\ exact values; all quantities computed on
    1000 equispaced points of the unit-circle orbit.
    The last row verifies the bound of
    Theorem~\ref{thm:liouville}\textup{(i)}:
    $d\,C(U)\,T=30.34\ge 4\pi=12.57=|\ln\det M_\gamma|$.}
  \label{tab:exp_D}
  \begin{tabular}{lcc}
    \hline
    Quantity & Numerical & Exact \\
    \hline
    $\Tr(Df_{\mathrm{SL}}(\gamma(t)))$
      & $-2.00000$ & $-2$ \\
    $\int_0^{2\pi}\Tr(Df_{\mathrm{SL}})\,dt$
      & $-12.56637$ & $-4\pi$ \\
    $\det M_\gamma = e^{\int\Tr}$
      & $3.4873\times10^{-6}$ & $e^{-4\pi}$ \\
    $|\ln\det M_\gamma|$
      & $12.566$ & $4\pi=12.566$ \\[2pt]
    \hline\noalign{\vskip 2pt}
    Thm.~\ref{thm:liouville} bound $d\,C(U)\,T$
      & $30.34$ & $\ge 4\pi\;\checkmark$ \\
    \hline
  \end{tabular}
\end{table}

{\tolerance=800\emergencystretch=1.5em
\paragraph{Illustration~F: Individual Floquet multiplier bounds
(Theorem~\ref{thm:individual}).}
Using the same bias-shifted $256$-unit MLP from Illustration~C,
we solve the variational equation
$\dot\Psi=Df_\theta(\gamma(t);s)\,\Psi$, $\Psi(0)=I$, along the
unit-circle orbit for each~$s$, obtaining the transition matrix
$\Psi(T)$.
Note: the reference orbit $\gamma$ (the unit circle of the
Stuart--Landau system) is not generally a $T$-periodic orbit of
the \emph{learned} vector field $f_\theta$; consequently $\Psi(T)$
is the transition matrix of the linearised flow along the reference
curve, not the monodromy matrix in the strict Floquet-theory sense.
In particular, no trivial eigenvalue $\mu_0=1$ is enforced, which
accounts for the absence of a unit multiplier in
Table~\ref{tab:individual_multipliers}.
Nevertheless, the Gr\"onwall argument underlying
Theorem~\ref{thm:individual} applies to the transition matrix along
\emph{any} curve $\gamma\subset U$: the eigenvalue bounds
$e^{-C(U)T}\le|\mu_i|\le e^{C(U)T}$ remain valid for the
eigenvalues of $\Psi(T)$, and the illustration below is a
diagnostic verification of those bounds along the reference orbit.
Table~\ref{tab:individual_multipliers} reports the two eigenvalues
$\mu_1,\mu_2$ of $\Psi(T)$ and the bounds of
Theorem~\ref{thm:individual}:
$e^{-C(U)\,T}\le|\mu_i|\le e^{C(U)\,T}$.
The bounds hold in every case;
Figure~\ref{fig:individual_multipliers} visualises how the
allowed window narrows with increasing~$s$.
As $\delta\to 0$ the window
$(e^{-C(U)T},\,e^{C(U)T})$ collapses to $\{1\}$, squeezing
both eigenvalues toward unity and eliminating both strong
contraction and expansion.\par}

\begin{table}[ht]
  \centering
  \caption{%
    \textbf{Individual Floquet multiplier bounds
    (Theorem~\ref{thm:individual}).}
    Bias-shifted $256$-unit tanh MLP
    (Illustration~C protocol).
    The transition matrix $\Psi(T)$ along the
    reference Stuart--Landau orbit is computed via the variational
    equation for each~$s$; $\mu_1,\mu_2$ are its eigenvalues
    (no trivial $\mu_0=1$ is enforced since $\gamma$ is not a
    periodic orbit of $f_\theta$).
    As $\delta\to 0$, both $|\mu_i|\to 1$.}
  \label{tab:individual_multipliers}
  \smallskip
  \begin{tabular}{@{}r c c c c c c@{}}
    \hline
    $s$ & $\delta$ & $C(U)$ & $|\mu_1|$ & $|\mu_2|$
        & $e^{-C(U)T}$ & $e^{C(U)T}$ \\
    \hline
    1 & 0.526 & 2.223
      & $2.30\times10^{-5}$ & $1.18$
      & $8.58\times10^{-7}$ & $1.17\times10^{6}$ \\
    2 & 0.128 & 0.776
      & $7.86\times10^{-2}$ & $7.86\times10^{-2}$
      & $7.64\times10^{-3}$ & $1.31\times10^{2}$ \\
    4 & 0.005 & 0.024
      & $0.951$ & $0.962$
      & $0.858$ & $1.165$ \\
    7 & ${<}10^{-4}$ & ${<}10^{-3}$
      & $1.000$ & $1.000$
      & $0.999$ & $1.001$ \\
    \hline
  \end{tabular}
\end{table}

\begin{figure}[ht]
  \centering
  \includegraphics[width=0.72\linewidth]{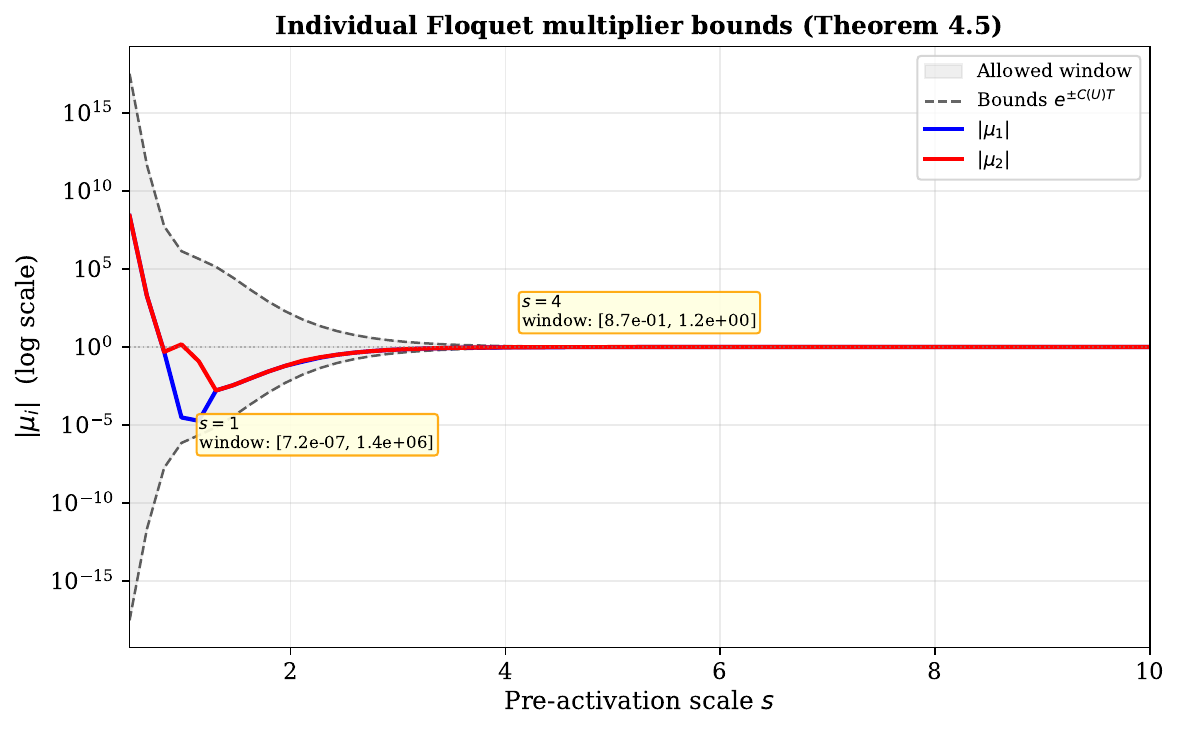}
  \caption{%
    \textbf{Individual Floquet multiplier bounds
    (Theorem~\ref{thm:individual}).}
    The grey band is the allowed window
    $[e^{-C(U)T},\,e^{C(U)T}]$; blue and red curves are
    $|\mu_1|,|\mu_2|$.
    As $s$ increases ($\delta\to 0$), the window collapses to
    $\{1\}$ and both multipliers converge to unity.}
  \label{fig:individual_multipliers}
\end{figure}

\paragraph{Illustration~E: Refined bound verification
(Theorems~\ref{thm:single}--\ref{thm:multi}).}
We verify the refined bound of
Section~\ref{sec:refined} on the same Stuart--Landau benchmark.
A $32$-unit $\tanh$ MLP is trained at $s=1$; weights are frozen
and $s$ is swept from~$1$ to~$30$.

\begin{table}[ht]
\centering
\caption{%
  \textbf{Bound comparison at $h^\star$
  (Illustration~E).}
  Ratio $=$ original\,/\,refined.
  The refined bound tracks the actual norm closely
  (within ${\sim}5\%$ at strong saturation),
  while the original bound overestimates by up to~$9\times$.}
\label{tab:exp_E}
\smallskip
\begin{tabular}{@{}r r r r r r@{}}
\hline
$s$ & Actual & Refined & Original & Ratio & $\delta$ \\
\hline
1.0  & 1.612  & 2.125  & 2.482  & 1.2$\times$ & 0.993 \\
4.2  & 3.325  & 3.913  & 9.342  & 2.4$\times$ & 0.885 \\
7.4  & 3.210  & 3.636  & 12.935 & 3.6$\times$ & 0.695 \\
10.7 & 2.564  & 2.846  & 13.164 & 4.6$\times$ & 0.494 \\
13.9 & 1.845  & 2.016  & 11.330 & 5.6$\times$ & 0.326 \\
17.1 & 1.252  & 1.350  & 8.796  & 6.5$\times$ & 0.206 \\
20.3 & 0.821  & 0.876  & 6.391  & 7.3$\times$ & 0.126 \\
23.6 & 0.527  & 0.557  & 4.443  & 8.0$\times$ & 0.075 \\
26.8 & 0.334  & 0.351  & 2.998  & 8.5$\times$ & 0.045 \\
30.0 & 0.210  & 0.219  & 1.980  & 9.0$\times$ & 0.026 \\
\hline
\end{tabular}
\end{table}

\begin{figure}[ht]
  \centering
  \includegraphics[width=\linewidth]{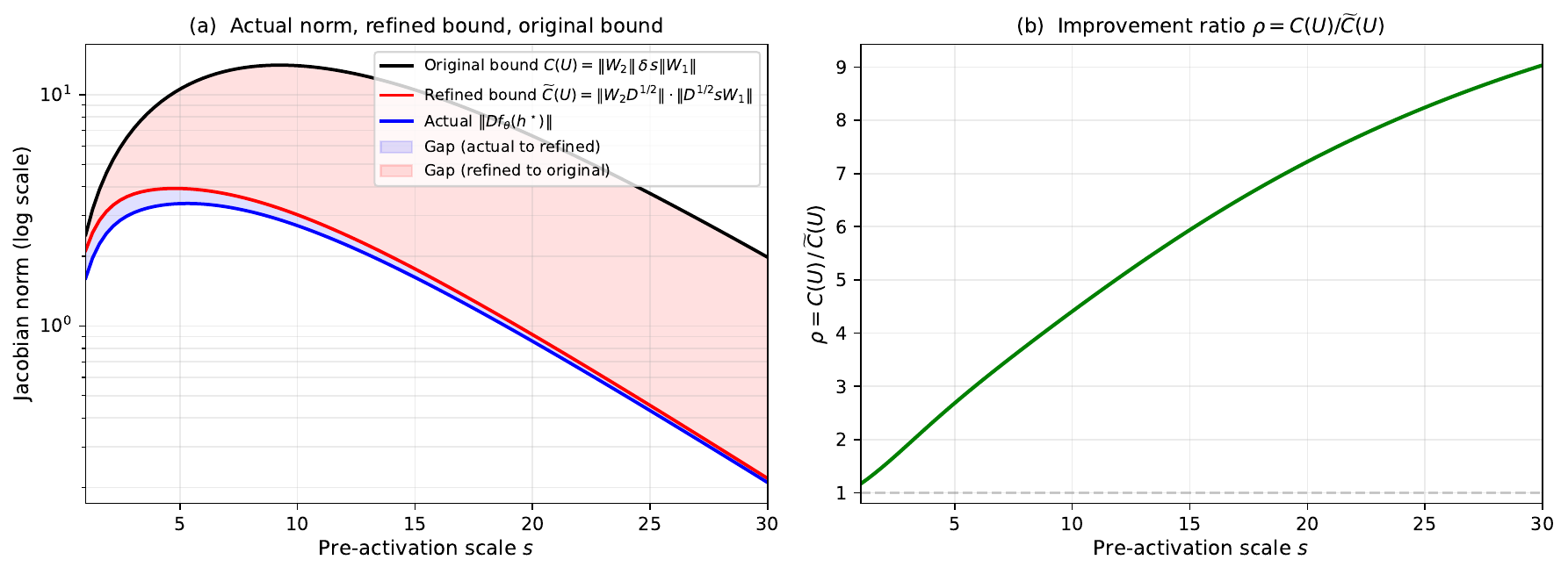}
  \caption{%
    \textbf{Refined bound verification (Illustration~E).}
    A $32$-unit $\tanh$ MLP is trained at $s{=}1$; weights are frozen
    and $s$ is swept from~$1$ to~$30$.
    \emph{(a)}~Actual Jacobian norm, refined bound,
    and original bound on a log scale.
    \emph{(b)}~Improvement ratio (original\,/\,refined)
    grows monotonically to~$9\times$ at $s=30$.}
  \label{fig:exp_E}
\end{figure}

Table~\ref{tab:exp_E} and Figure~\ref{fig:exp_E} confirm
$\text{actual}\le\text{refined}\le\text{original}$
throughout.
The improvement ratio reaches $9\times$ at $s=30$;
the refined-to-actual gap shrinks to ${\sim}4\%$ at strong saturation.

\clearpage
\bibliographystyle{plainnat}
\bibliography{refs}

\end{document}